\documentclass[11pt]{article}

\usepackage{amsmath,amssymb,amsthm}
\usepackage{color}
\usepackage{graphicx}
\newtheorem{theorem}{Theorem}

\newtheorem{remark}{Remark}

\def \beq{ \begin{equation}}
\def \eeq{\end{equation}}

\title{Equal masses Eulerian relative equilibria on a rotating meridian of $\mathbb{S}^2$}

\date{}

\begin{document}
\maketitle
\author{
\begin{center}
	{Toshiaki~Fujiwara$^1$, Ernesto P\'{e}rez-Chavela$^2$}\\	
		\bigskip
	   $^1$College of Liberal Arts and Sciences, Kitasato University, Japan\\
	   fujiwara@kitasato-u.ac.jp\\
	    $^2$Department of Mathematics, ITAM, M\'exico\\
	    ernesto.perez@itam.mx\\
	\end{center}
}
	
\begin{abstract}
Relative equilibria on a rotating meridian on $\mathbb{S}^2$
in equal-mass three-body problem 
under the cotangent potential are  determined.
We show the existence of scalene and isosceles 
relative equilibria.

Almost all isosceles triangles, including equilateral, can form a relative equilibrium,
except for the two equal arc angles $\theta = \pi/2$.
For $\theta\in (0,2\pi/3)\setminus \{\pi/2\}$,
the mid mass must be on
the rotation axis, in our case, at the north or south pole of $\mathbb{S}^2$.

For $\theta\in (2\pi/3,\pi)$, the mid mass must be on the equator.
For $\theta=2\pi/3$, we obtain the equilateral triangle, 
where the
position of the masses is arbitrary.

When the largest arc angle $a_\ell$ is in $a_\ell\in (\pi/2,a_c)$, with $a_c=1.8124...$,
two scalene configurations 
exist for given $a_\ell$.
\end{abstract}



\section{Introduction}
Relative equilibria are the simplest solutions in the three body problem ($n$--body in general), where the masses move uniformly in a circular motion, as if they formed a rigid body. In other words where the forces produced by the rotation are in perfect balance with the attractive forces among the masses. For a nice overview about relative equilibria in Euclidean spaces see \cite{Wintner}.

In a recent work, we develop
a systematic method
to study relative equilibria on $\mathbb{S}^2$
in the three-body problem \cite{arxiv1, arxiv2}.
This method is applicable to investigate the Euler configurations
(relative equilibria where the three bodies are on a geodesic)
and the extended Lagrange configurations (relative equilibria where
the three bodies are not on a geodesic) with general masses.
In \cite{arxiv1}, we totally solved the case when the three bodies are on the equator. Additionally we explained why is not possible to have Eulerian relative equilibria on a geodesic other than the equator or a meridian. For this reason, here we concentrate on the analysis of relative equilibria on a rotating meridian for the equal masses case. As you will notice, this is not a trivial case, and allow us to clarify our method and verify that everything works well.

The relative equilibria on a  rotating meridian 
are the motions 
where the angle from the north pole is fixed, that is
$\theta_k(t)=\theta_k(0)$
and their longitude $\phi_k$ on a 
rotating meridian
with constant angular velocity $\omega$ is,
$\phi_k(t)=\omega t$ for $k=1,2,3$.
See Figure \ref{figThreeTypicalREOnARotatingMeridian}.
\begin{figure}
   \centering
   \includegraphics[width=4.3cm]{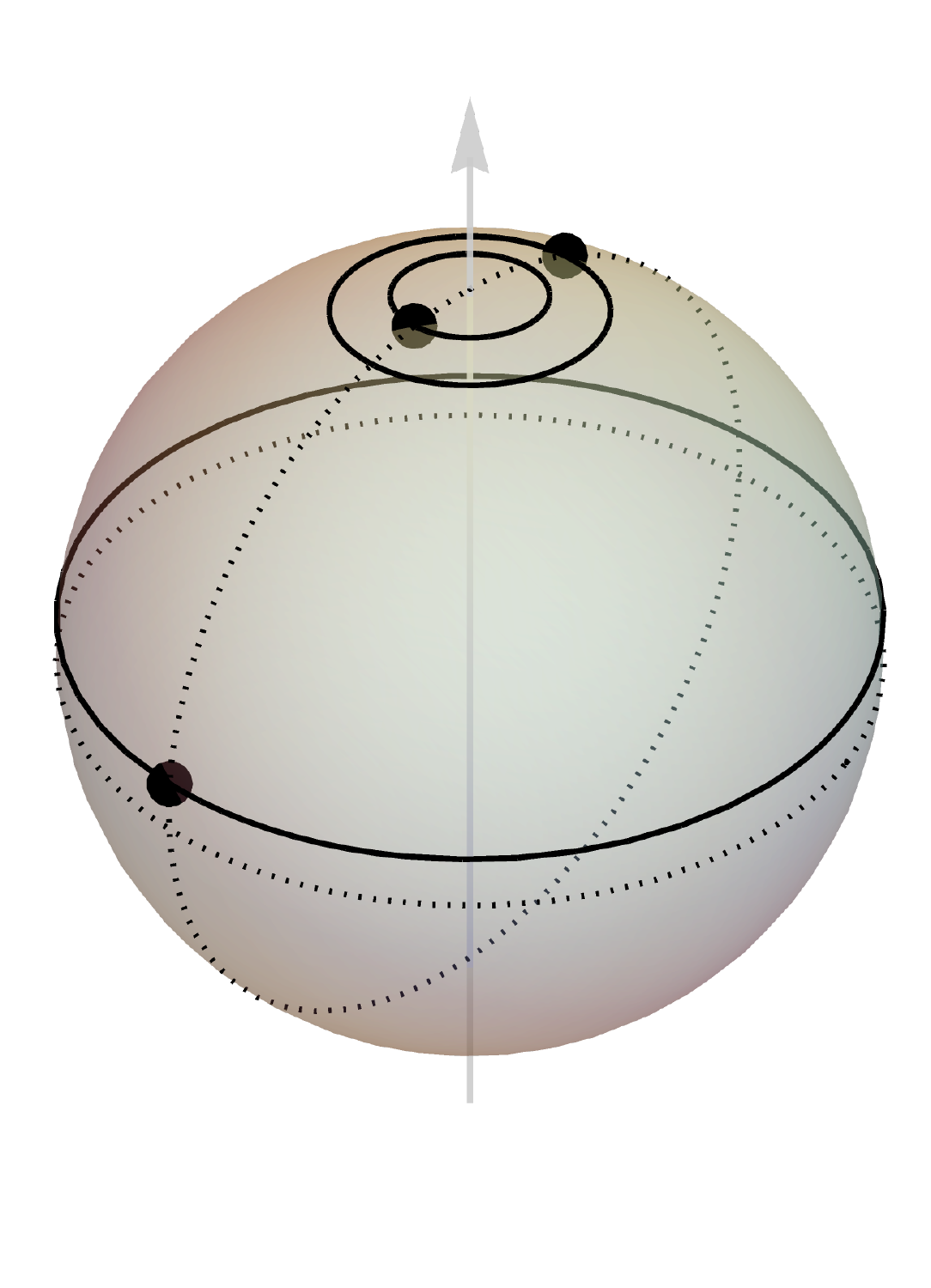}%
   \includegraphics[width=4.3cm]{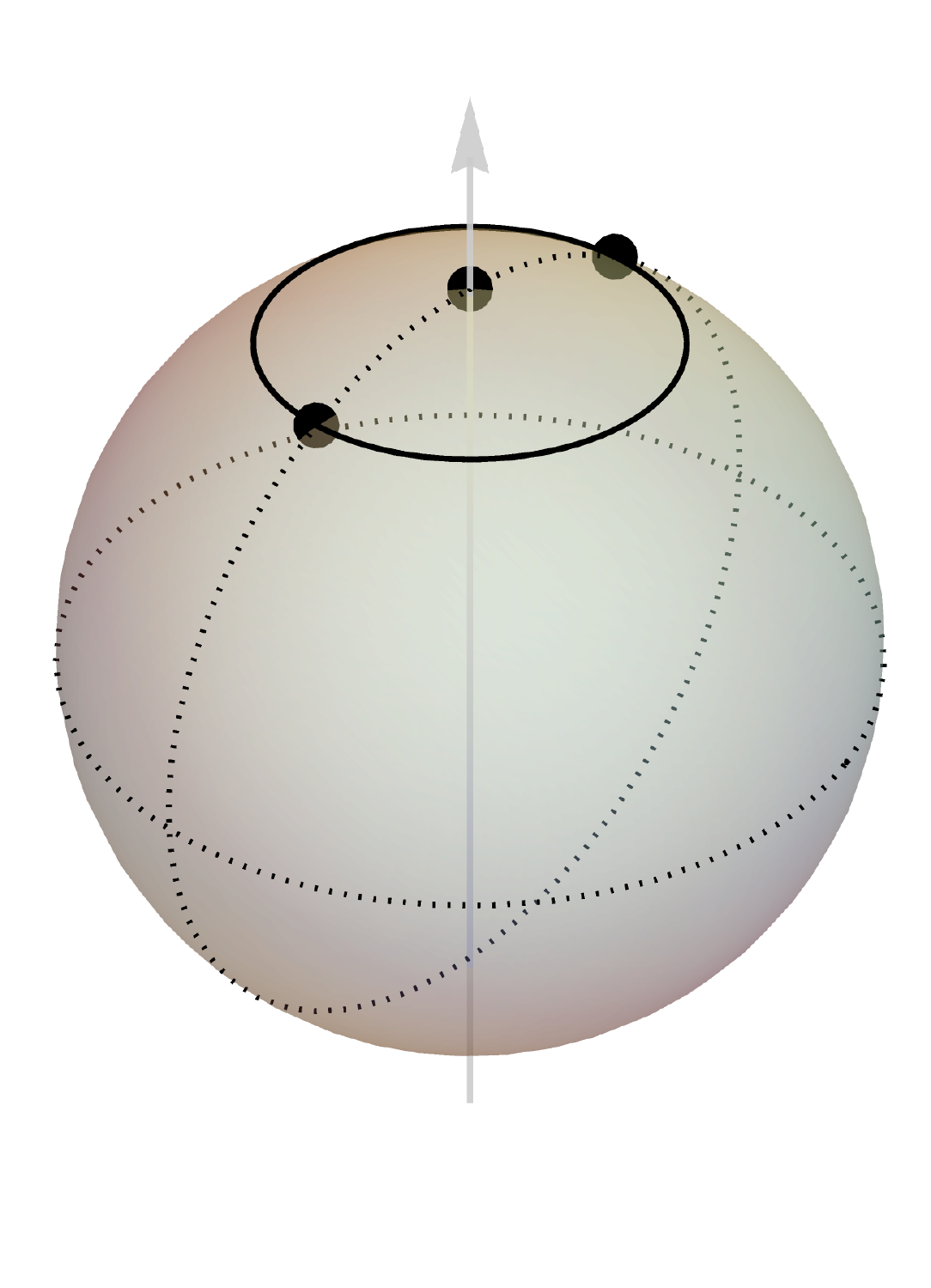}%
   \includegraphics[width=4.3cm]{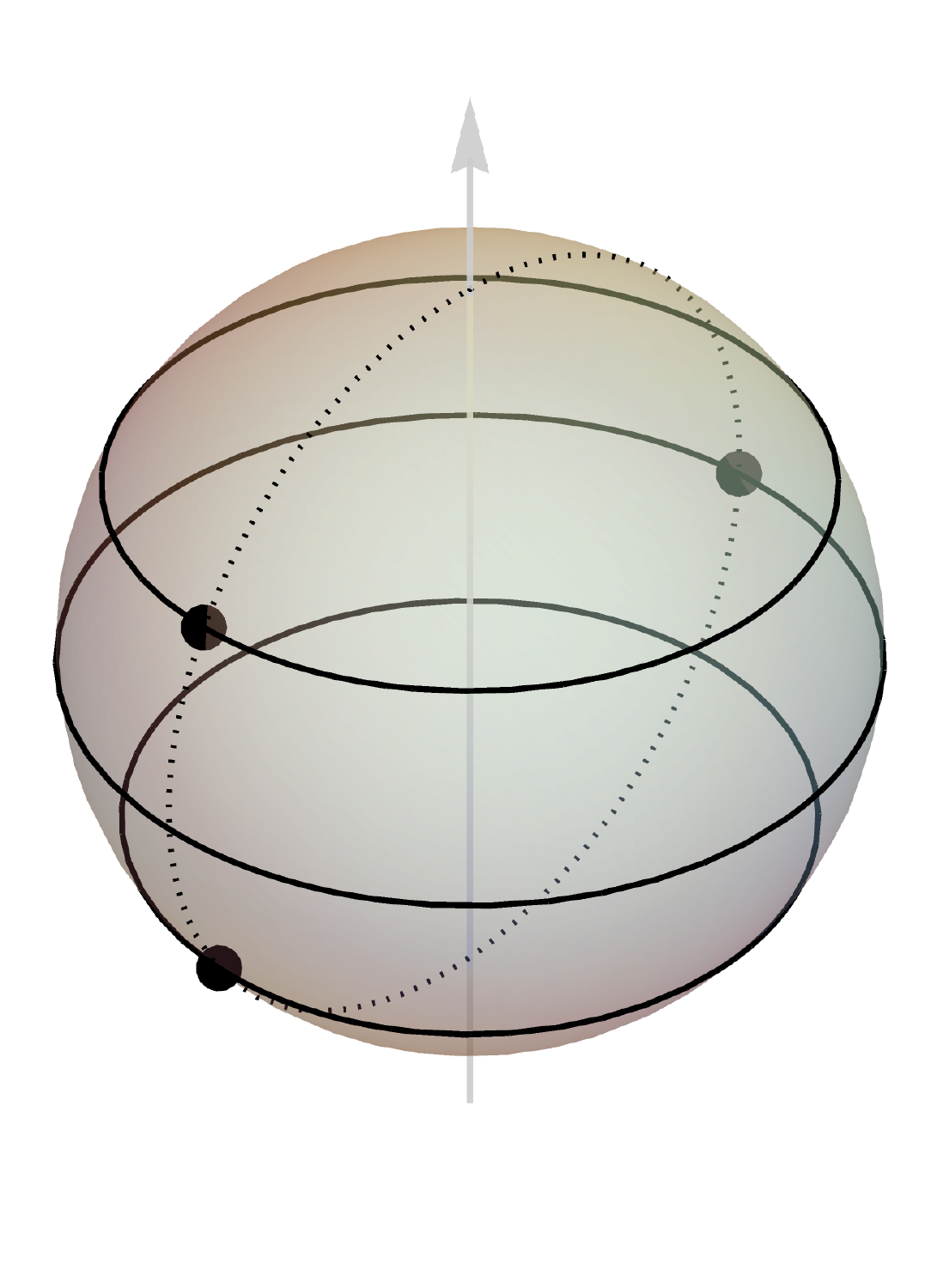}%
   \caption{Three typical relative equilibria on a rotating meridian
   for equal-mass three-body problem on $\mathbb{S}^2$ are shown.
   Three bodies are rotating around the $z$-axis
   as if it were a rigid body.
   Three black balls represent three masses.
   Solid circles represent the orbit of masses.
   Two dotted circles and the arrow represent the rotating meridian, the Equator,
   and the the $z$-axis, respectively.
   Left: A scalene configuration, where three arc angles between bodies are different.
   The largest arc angle $a_\ell$ satisfies $\pi/2<a_\ell<a_c=1.8124...$.
   Middle: An isosceles configuration with two equal arc angles $\theta$ are smaller than $2\pi/3$
   and $\theta\ne \pi/2$, 
   in this case the mid mass must be on one of the pole. 
   When $\theta=\pi/2$, the two masses are on the antipodal point of each other,
   where the equations of motion are not defined.
   Right: An isosceles configuration with $2\pi/3<\theta < \pi$, 
   in this case the mid mass must rotates on the equator.
   }
   \label{figThreeTypicalREOnARotatingMeridian}
\end{figure}

We also restricted our analysis to the cotangent potential given by

\begin{equation}\label{cotangent pot}
U=\frac{\cos\theta_{ij}}{\sqrt{1-\cos^2(\theta_{ij})}},
\end{equation}
where $\theta_{ij}=\theta_i-\theta_j$,
and $\theta_k$, $k=1,2,3$ is the angle from the north pole,
$-\pi<\theta_k<\pi$.
The cotangent potential is the potential used in the analysis of the $n$--body problem, when the masses move on a surface of constant positive curvature. So, the results obtained in this article can be presented as new families of relative equilibria for the positive curved $3$--body problem \cite{Diacu-EPC1, Diacu1, Diacu4, EPC1, zhu}.

Without loss of generality, along this paper
we take $m_k=1$ and the radius of $\mathbb{S}^2$ is $1$.
Since
\begin{equation}
\frac{dU}{d\theta_{ij}}=-\frac{\sin\theta_{ij}}{|\sin\theta_{ij}|^3}<0
\end{equation}
for $0<\theta_{ij}<\pi$, this potential produces attractive force among bodies.

It is important to distinguish {\it the shape} and the {\it configuration.}
The shape is described by two mutual angles,
for example $\theta_{21}$ and $\theta_{31}$.
On the other hand, the configuration is described by the three angles $\theta_k$.
The map from a configuration to the shape is trivial.
However, the map from a shape to the configuration needs 
additional information.
This information is provided by the angular momentum $c_x=c_y=0$,
which is reduced to one equation 
\begin{equation}
\sum_k \sin(2\theta_k)=0
\mbox{ if }\omega\ne 0
\end{equation}
for the system on a rotating meridian.
In Section \ref{secTranslation},
we review the formula (hereafter referred to as ``translation formula'') for the map
from a shape to the corresponding configuration.

Utilising this formula,
the equations of motion are converted to the conditions for the shape.
If a shape satisfies this condition  we call it a ``rigid rotator'',
because it rotates as if it were a rigid body.
We review this step in section \ref{conditionForRotator}. 
Once we find a ``rigid rotator'',
it is translated to the corresponding configuration (relative equilibrium)
by the translation formula. 

In the same section (section \ref{conditionForRotator}),
we will show that the ``rigid rotators'' (shape of the relative equilibria)
for equal masses case are scalene or isosceles triangles (including equilateral).
As far as we know, this is the first time that Eulerian scalene relative equilibria, for  equal masses are introduced.
Using a different approach, S. Zhu proved the existence of acute and obtuse triangles on a rotating meridian, in particular scalene triangles, but with not all  equal masses \cite{zhu}. For the case of isosceles triangles, we have recovered Zhu's results by using our method .

Let $a$ be one angle between a pair of positions (for instance $\theta_{21}$).
For given $0<a<a_c=1.8124...$,
there are two scalene relative equilibria.
See  Figure \ref{figThreeTypicalREOnARotatingMeridian}.
The scalene relative equilibria will be treated in Section~\ref{secScalane}.
The exact value of $\cos(a_c)$ will be shown there.
It will be also shown that the largest arc angles $a_\ell$ between bodies
must be in the interval
$ (\pi/2,a_c)$ 
in order to have a scalene equilibria.
Two scalene equilibria are shown in the Figure \ref{figEqualMassesAeqPiDiv6RelativeEquilibria},
where $a=\theta_{21}=\pi/6$.
Some exact values for the scalene equilibria with $\cos(\theta_{21})=-1/8$
are  shown in the appendix \ref{secExactValues}.
\begin{figure}
   \centering
   \includegraphics[width=12cm]{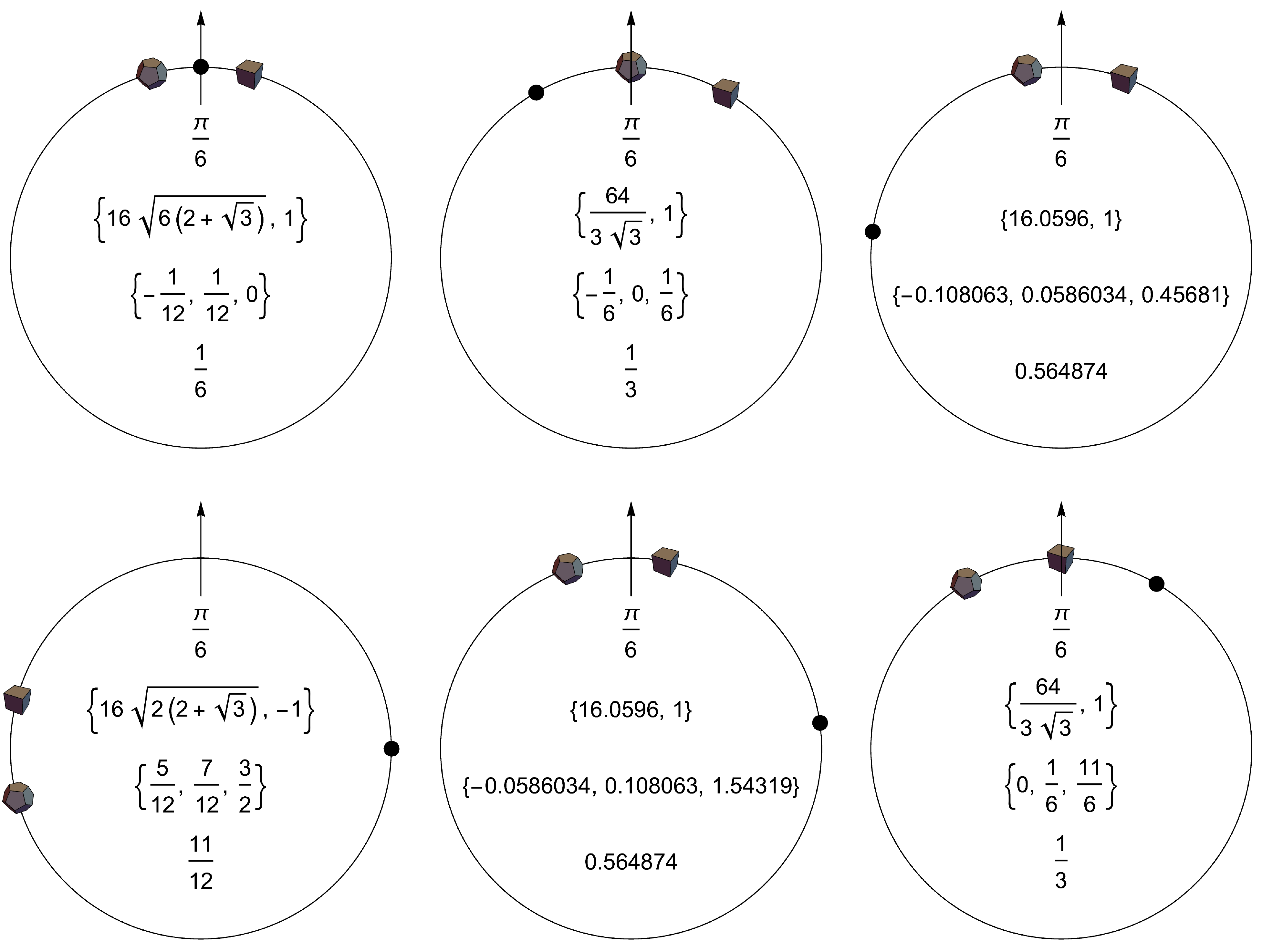}
   \caption{Six relative equilibria for
   equal masses $a=\theta_{21}=\pi/6 < a_c$.
   The cube, dodecahedron, and ball represent
   $m_1$, $m_2$, and $m_3$ respectively.
   The arrow represents the $z$-axis.
   We will continue using this convention 
   for the masses and $z$-axis in the following figures.
   As we can see, there are
   two scalene configurations
   (upper right and lower middle),
   and four isosceles configurations.
   In each image, the lines represent respectively
   $a=\pi/6$,
   $\{\omega^2, s\}$,
   $\{\theta_1/\pi,\theta_2/\pi,\theta_3/\pi\}$,
   and the largest arc angle $a_\ell/\pi$.
   The two scalene triangles are congruent
   with exchanging the masses $m_1$ and  $m_2$
   and $a_\ell/\pi=0.5648...>1/2$.}
   \label{figEqualMassesAeqPiDiv6RelativeEquilibria}
\end{figure}

For any given $a \in (0,\pi)\backslash\{\pi/2\}$,
there always exist four isosceles relative equilibria.
Namely, almost all isosceles triangle can form a relative equilibrium.
The isosceles relative equilibria will be treated in Section \ref{secIsosceles}.
Let $0<\theta<\pi$ be the equal angle of the isosceles triangle,
and the mass $3$ is in the mid point,
that is $\theta=\theta_2-\theta_3=\theta_3-\theta_1$.
We will show in Section \ref{secIsosceles} that
$\theta_3=0 \mod (\pi)$ for
$\theta\in (0,2\pi/3)\backslash\{\pi/2\}$,
and $\theta_3=\pi/2 \mod (\pi)$ for $\theta\in (2\pi/3,\pi)$. 
See Figure \ref{figThreeTypicalREOnARotatingMeridian}. 
The above result was proved previously by S. Zhu by using different techniques (see \cite{zhu} for more details).
For $\theta=2\pi/3$, the shape of the configuration is equilateral,
$\omega=0$, and the angles $\theta_k$ are undetermined.
Only the mutual angles are determined to be $2\pi/3$.

Finally Section \ref{secFinalRemarks} is devoted to the final remarks.


\section{Correspondence between a configuration and a shape}
\label{secTranslation}
In this section, we give
the equations of motion for the equilibria on a rotating meridian
and the translation formula which translates the shape to the configuration.

The equations of motion for the equilibria on a rotating meridian
are given by
\begin{equation}
\label{equationsOfMotion}
\frac{\omega^2}{2}\sin(2\theta_k)
=\sum_{i\ne k}\frac{\sin(\theta_k-\theta_i)}{|\sin(\theta_k-\theta_i)|^3},
\end{equation}
where $k$ and $i \in \{1,2,3\}$.
The sum of the equations for $k=1,2,3$ yields
\begin{equation}
\label{CxCy}
\omega^2 \sum_k \sin(2\theta_k)=0.
\end{equation}
This is a first integral,
which corresponds to the angular momentum $c_x=c_y=0$.
It is clear that the condition (\ref{CxCy}) is a necessary condition for $\theta_k$
to satisfy the equations of motion.

Using this equation,
we can find the translation formula.
Let $\theta_{ij}=\theta_i-\theta_j$.
Then the equation (\ref{CxCy}) for $\omega^2\ne 0$ is
\begin{equation}
\label{conditionForTheta1}
\begin{split}
0
&=\sin(2\theta_1)+\sin(2(\theta_1+\theta_{21}))+\sin(2(\theta_1+\theta_{31}))\\
&=\sin(2\theta_1)\Big(1+\cos(2\theta_{21})+\cos(2\theta_{31})\Big)
	+\cos(2\theta_1)\Big(\sin(2\theta_{21})+\sin(2\theta_{31})\Big)\\
&=A\sin(2\theta_1+2\alpha).
\end{split}
\end{equation}
Where,
\begin{equation}
\begin{split}
A
&=\left(
	\Big(1+\cos(2\theta_{21})+\cos(2\theta_{31})\Big)^2
	+\Big(\sin(2\theta_{21})+\sin(2\theta_{31})\Big)^2
	\right)^{1/2}\\
&=\left(3+2\Big(
		\cos(2\theta_{12})+\cos(2\theta_{23})+\cos(2\theta_{31})
		\Big)
	\right)^{1/2},
\end{split}
\end{equation}
and 
\begin{equation}
\begin{split}
\cos(2\alpha)
&=A^{-1}\Big(1+\cos(2\theta_{21})+\cos(2\theta_{31})\Big),\\
\sin(2\alpha)
&=A^{-1}\Big(\sin(2\theta_{21})+\sin(2\theta_{31})\Big),
\end{split}
\end{equation}
if $A\ne 0$.
The case of $A=0$ will be considered in the last of this section.
Let us proceed assuming $A\ne 0$.
The solution of (\ref{conditionForTheta1}) is
$2\theta_1=-2\alpha$ or $-2\alpha+\pi$.
Namely,
\begin{equation}
\label{tranlationFormula}
\begin{split}
\cos(2\theta_1)
&=sA^{-1}\Big(1+\cos(2\theta_{12})+\cos(2\theta_{13})\Big),\\
\sin(2\theta_1)
&=sA^{-1}\Big(\sin(2\theta_{12})+\sin(2\theta_{13})\Big).
\end{split}
\end{equation}
Where $s=\pm 1$.

Although there are ambiguity for $\theta_1$ modulo $\pi$,
the configuration with $\theta_1$ and $\theta_1+\pi$
is just  
a reflection on the equator (upside-down  or north-south each other).

This equation determines the configuration variable $\theta_1$, 
through 
the two shape variables $\theta_{21}$ and $\theta_{31}$.
The other angles are determined by
$\theta_2=\theta_1+\theta_{21}$ and $\theta_3=\theta_1+\theta_{31}$.
These are the translation formulas.

For a later use, let us describe the equations for the other angles
that are derived by (\ref{tranlationFormula}).
\begin{equation}
\label{tranlationFormula2}
\begin{split}
\sin(2\theta_2)&=sA^{-1}\Big(\sin(2\theta_{21})+\sin(2\theta_{23})\Big),\\
\sin(2\theta_3)&=sA^{-1}\Big(\sin(2\theta_{31})+\sin(2\theta_{32})\Big),\\
\cos(2\theta_2)&=sA^{-1}\Big(\cos(2\theta_{21})+1+\cos(2\theta_{23})\Big),\\
\cos(2\theta_3)&=sA^{-1}\Big(\cos(2\theta_{31})+\cos(2\theta_{32})+1\Big).
\end{split}
\end{equation}

Now, let us consider the case of $A=0$.
This will happen if 
$$1+\cos(2\theta_{21})+\cos(2\theta_{31})
=\sin(2\theta_{21})+\sin(2\theta_{31})
=0.$$
The solution is
\begin{equation}
\begin{split}
\cos(2\theta_{21})&=\cos(2\theta_{31})=-\frac{1}{2},\\
\sin(2\theta_{21})&=-\sin(2\theta_{31}).
\end{split}
\end{equation}

Two shapes satisfy this equation.
To make the description clear,
let us write $a=\theta_{21}$ and $x=\theta_{31}$.
We can restrict $0<a<\pi$ and $-\pi<x<\pi$ without loss of generality.
From now on, 
$a$ and $x$ represent the same angle and take the same range. 

One shape is equilateral triangle, $a=2\pi/3$ and $x=-2\pi/3$.
Then the right hand side of the equation (\ref{equationsOfMotion})
for $k=1,2,3$ is zero.
Therefore, the solution is $\omega=0$,
what is so called a fixed point.
When $\omega=0$, the angles $\theta_k$ are not determined.
Only the difference of the angles $\theta_{ij}$ has meaning.

Another solution is an isosceles triangle with
two equal angles equal to $\pi/3$.
The corresponding $a$ and $x$ are
$(a,x)=(\pi/3,2\pi/3)$, $(\pi/3,-\pi/3)$, and $(2\pi/3,\pi/3)$.
Actually, we can convince us that
$\sin(2\theta_1)+\sin(2(\theta_1+\pi/3))+\sin(2(\theta_1+2\pi/3))
=\sin(2\theta_1)+\sin(2(\theta_1+\pi/3))+\sin(2(\theta_1-\pi/3))
=0$ is an identity for all $\theta_1$.
However, the equations (\ref{equationsOfMotion})  for $(a,x)=(2\pi/3,\pi/3)$ are
\begin{equation}
\begin{split}
-\omega^2\sin(2\theta_1)&=\omega^2\sin(2\theta_2)=8/3,\\
\omega^2\sin(2\theta_3)&=0.
\end{split}
\end{equation}
From the last line, we get $2\theta_3=0$ or $\pi$.
Then, from the first line,
we get $\theta_3=0$ and $\theta_2=-\theta_1=\pi/3$, 
$\omega^2=16/(3\sqrt{3})$.
Thus $\theta_k$ and $\omega^2$ are determined by the equations of motion
for this case.

This example clearly shows that
$A=0$ is not always the fixed point 
(relative equilibrium with $\omega^2=0$).
The equations of motion (\ref{equationsOfMotion}) determines
whether the shape is a fixed point or not.


\section{Condition for a shape to be a rigid rotator}
\label{conditionForRotator}
In this section, assuming that $A\ne 0$,
we rewrite the equations of motion to obtain the conditions for a shape.
If a shape satisfies this condition, the shape can form a relative equilibrium.

Using the translation formulae (\ref{tranlationFormula}--\ref{tranlationFormula2}),
the equations of motion (\ref{equationsOfMotion}) can be written as
\begin{equation}
\label{condition1}
\frac{s\omega^2}{2A}
	\Big(\sin(2\theta_{12})+\sin(2\theta_{13})\Big)
=\frac{\sin(\theta_{12})}{|\sin(\theta_{12})|^3}
	+\frac{\sin(\theta_{13})}{|\sin(\theta_{13})|^3},
\end{equation}
and similar equations.
Now let 
\begin{equation}
G_{ij}=\sin(2\theta_{ji}),\qquad F_{ij}=\frac{\sin\theta_{ji}}{|\sin\theta_{ji}|^3}.
\end{equation}
Then the equations of motion (\ref{equationsOfMotion}) are equivalent to
 \begin{equation}
 \label{theConditions}
 \begin{split}
 \frac{s\omega^2}{2A}(G_{12}-G_{23})&=F_{12}-F_{23},\\
\frac{s\omega^2}{2A}(G_{23}-G_{31})&=F_{23}-F_{31},\\
\frac{s\omega^2}{2A}(G_{31}-G_{12})&=F_{31}-F_{12},
 \end{split}
 \end{equation}
 if $A\ne 0$.
 Only two of the above conditions
 are independent.
In this paper, the first and the last one will be used. 
We have the following result.

\begin{theorem}
[Condition for a shape]
\label{propConditionForShape}
If $A\ne0$,
\begin{equation}
f
=\left|\begin{array}{cc}
G_{12}-G_{23} & G_{31}-G_{12} \\
F_{12}-F_{23} & F_{31}-F_{12}
\end{array}\right|
=0
\end{equation}
is a necessary and sufficient condition for a shape to satisfy the equations of motion.
\end{theorem}

\begin{proof}
We will show that $f=0$ is equivalent to the equations (\ref{theConditions}).

If equations (\ref{theConditions}) are satisfied,
\begin{equation}
f=\frac{s\omega^2}{2A}
\left|\begin{array}{cc}
G_{12}-G_{23} & G_{31}-G_{12} \\
G_{12}-G_{23} & G_{31}-G_{12}
\end{array}\right|
=0
\end{equation}
is obvious.

Inversely, if $f=0$, there are two cases.

The first case is when all elements of the matrix are zero.
For this case, the equations of motion (\ref{theConditions}) are trivially satisfied
and $\omega$ is undetermined.
We can show that this case only happens when the shape is equilateral.
Because,
we get $\sin(2a)=\sin(2(a-x))=-\sin(2x)$ from the equations for $G$,
and $\sin a =\sin(a-x))=-\sin x$ from the equations for $F$.
This yields $\cos a=\cos(a-x)=\cos x=-1/2$.
The solution in $0<a<\pi$, and $-\pi<x<\pi$ and $a\ne x$ is
$a=-x=2\pi/3$, namely the shape is equilateral.
Then, the original equations of motion (\ref{equationsOfMotion}) are satisfied by 
$\omega=0$ and $\theta_k$ are undetermined.

The second case 
is when at least one of the elements of the matrix 
is not zero.
For example, let be $G_{12}-G_{23}\ne 0$.
Then we can define 
\begin{equation}
\frac{s\omega^2}{2A}=\frac{F_{12}-F_{23}}{G_{12}-G_{23}}.
\end{equation}
Then,
$f=0$ yields $s\omega^2/(2A)(G_{31}-G_{12})=F_{31}-F_{12}$.
Thus the equations  (\ref{theConditions}) are satisfied.
Similarly, if $F_{12}-F_{23}\ne 0$
then we can define
\begin{equation}
\frac{2A}{s\omega^2}=\frac{G_{12}-G_{23}}{F_{12}-F_{23}}.
\end{equation}
In this case $f=0$ yields
$G_{31}-G_{12}=(2A)/(s\omega^2)(F_{31}-F_{12})$ and then
 the equations  (\ref{theConditions}) are satisfied.
\end{proof}

The explicit form of $f=0$ is given by
\begin{equation}
\label{defOfFandG}
\begin{split}
f=&\frac{g}{(\sin x |\sin x|)(\sin a |\sin a|)(\sin(x-a) |\sin(x-a)|)}=0, \quad \text{where}\\
g=
&\sin(x)|\sin(x)|\Big(\sin(2x)+\sin(2a)\Big)
	\Big(\sin(x-a)|\sin(x-a)|-\sin a |\sin a|\Big)\\
&-\sin(x-a)|\sin(x-a)|\Big(\sin(2a)-\sin(2(x-a))\Big)
	\Big(\sin a |\sin a|+\sin x |\sin x|\Big).
\end{split}
\end{equation}

\begin{figure}
   \centering
   \includegraphics[width=13cm]{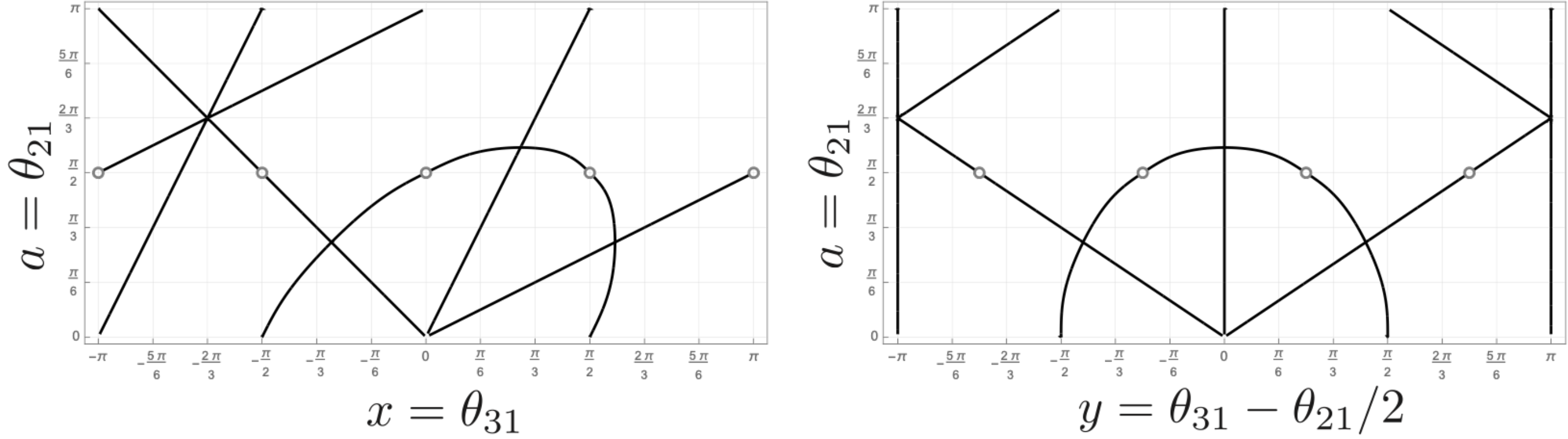}
   \caption{The contour for $f=0$.
   The vertical axis represents $a=\theta_{21}$.
   The horizontal axis is $x=\theta_{31}$ (left), and 
   $y=\theta_{31}-\theta_{21}/2$ (right).
   The vertical lines $y=\pm \pi$ represent the same line.
   The curve and the straight lines represent
   scalene and isosceles rigid rotators respectively.
   The four points 
   $(x,a)=(-\pi,\pi/2)=(+\pi,\pi/2)$, $(-\pi/2,\pi/2)$, $(0,\pi/2)$,
   $ (\pi/2,\pi/2)$
   and $(y,a)=(\pm\pi/4,\pi/2), (\pm 3\pi/4,\pi/2)$
   are excluded.}
   \label{figFbyXorY}
\end{figure}

The Figure \ref{figFbyXorY} shows the $f=0$ contours.
Note that the four points with $a=\pi/2$,
namely 
$(x,a)=(\pm \pi,\pi/2)$, $(-\pi/2,\pi/2)$, $(0,\pi/2)$, and $(\pi/2,\pi/2)$
are excluded,
because at $a=\pi/2$
\begin{equation}
f=\frac{2\sin x}{|\cos x|}-\frac{2\cos x}{|\sin x|}
\end{equation}
which has only two zeros at $x=\pi/4$ and $\pi/4-\pi$.

In the Figure \ref{figFbyXorY}, the straight lines represent the isosceles rigid rotator.
There are four lines,
$x=2a \mod (2\pi)$, $a/2$, $a/2-\pi$ and $-a$.
The reader can check that the number for isosceles triangles for given $a=\theta_2-\theta_1$
is just four.
This means that almost all isosceles shape is a rigid rotator,
and that the curve for $0<a<a_c=1.8124...$ represents scalene triangle.
See Figure \ref{figSixRigidRotatorsForAequalPiDiv6}.
\begin{figure}
   \centering
   \includegraphics[width=5cm]{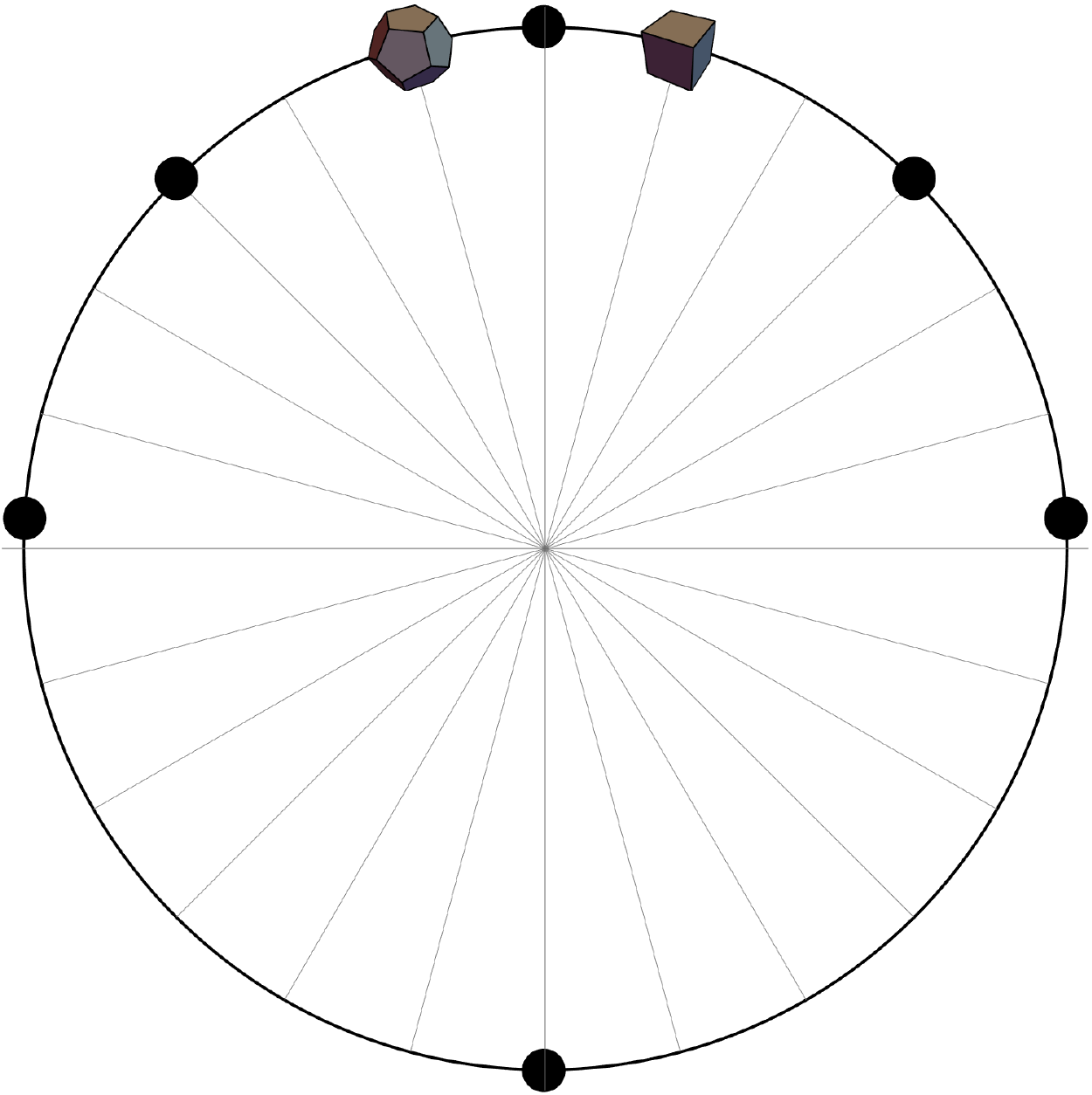} 
   \caption{Six rigid rotators (shapes that can satisfy the equations of motion)
   for $a=\theta_2-\theta_1=\pi/6$ in one picture.
   Only the angles between the masses have meaning.
   The grey straight lines are drawn every $\pi/12$.
   Black balls near the horizontal line represent scalene triangles.}
   \label{figSixRigidRotatorsForAequalPiDiv6}
\end{figure}

The scalene triangle will be treated in the Section \ref{secScalane},
and the isosceles triangle will be tackled in Section \ref{secIsosceles}.

Before closing this section, let us continue to observe the Figure \ref{figFbyXorY}.
The crossing point of the three straight lines is $(a,x)=(2\pi/3,-2\pi/3)$,
which represents the equilateral triangle. 
This shape is the fixed point.

The function $f$ and $g$ are invariant under the exchange of 
$a$ and $x$,
although this may not be so obvious from the equation (\ref{defOfFandG})
and  figure \ref{figFbyXorY}.
The invariance is because, the exchange of $a$ and $x$ is merely
the exchange of the places of $m_2$ and $m_3$,
and this exchange has no effect for the mass distribution and equations of motion.
Similarly, $f$ and $g$ are invariant under the map
$a\to -a$ and $x \to x-a$ which corresponds to 
the exchange of $m_1$ and $m_2$.

The right diagram of Figure \ref{figFbyXorY}
represents the contour of $f=0$ by the variables $a$ (vertical)
and $y=x-a/2$ (horizontal).
This diagram clearly shows the invariance of $f=0$ for $y \leftrightarrow -y$,
which corresponds to the exchange of $m_1$ and $m_2$
with $\pi$ rotation around the $z$-axis.
See Figure \ref{figSixRigidRotatorsForAequalPiDiv6}.
This symmetry will be used to investigate the scalene triangles in the next section.


\section{Scalene relative equilibria}
\label{secScalane}
To investigate the scalene relative equilibria
it is enough to look for the curve in Figure \ref{figFbyXorY}
defined on the region $0<x<a$ and equivalently $-a/2<y=x-a/2<a/2$.
Obviously, the edge of this region is $(x,a)=(0,\pi/2)$ and $(\pi/2,\pi/2)$.
In this region $a$ is the largest arc angle,
because $0<x=\theta_{31}<a=\theta_{21}(<\pi/2)$ and 
$0<\theta_{23}=a-x<a$.

The scalene triangle in the other region is the same shape
with different mass name by the invariance for $a \leftrightarrow x$
and $y \leftrightarrow -y$.

The function $g$ in this region is,
\begin{equation}
\begin{split}
g
=&2 \cos (x) \sin ^2(a-x) \sin (2 a-x) \Big(\sin ^2(a)+\sin ^2(x)\Big)\\
	&-\sin ^2(x) \Big(\sin ^2(a-x)+\sin ^2(a)\Big) \Big(\sin (2 a)+\sin (2 x)\Big)\\
=&\frac{\sin(2y)}{4}h,
\end{split}
\end{equation}
with
\begin{equation}
\begin{split}
h=&
	-\cos(a)\cos(4y)
	+2\Big(2\cos(2a)+\sin^2(2a)\Big)\cos(2y)\\
	&-\cos(a)\Big(\cos(4a)-5\cos(2a)+7\Big).
\end{split}
\end{equation}

Note that we factor out $\sin(2y)$ which corresponds to $y=0$
vertical straight line.
Therefore, we expect that $h=0$ represents the curve
in the Figure \ref{figFbyXorY}
which corresponds to the scalene triangles.
Indeed, we can explicitly show that $h=0$ defines just one curve in the
$(x,a)$ plane.
Note that $h=0$ is the quadratic equation for $\cos(2y)$.
The solution is
\begin{equation}
\label{solOfCos2y}
\cos(2y)
=\cos(a)+\sin(a)\tan(a)\left(
	\cos(2a)
	+\sqrt{\cos^2(2a)-4\cos(2a)-4}
	\right).
\end{equation}
Since the absolute value of another branch is greater than $1$ for $\cos(2y)$,
the solution is unique.
Obviously, the equation (\ref{solOfCos2y}) express a single curve in $(x,a)$ plane.
Therefore $h=0$ represents the curve in the Figure \ref{figFbyXorY}.

Now, let us find the crossing point of the curve $h=0$ and the line $y=0$.
The value of  $a$ at this point is the solution of
\begin{equation}
h(y=0,a)
=-8\sin^4(a/2)\Big(
	\cos(3a)+6\cos(2a)+14\cos(a)+8
	\Big)
=0.
\end{equation}
The solution $a_c\ne 0$ is
\begin{equation}
\begin{split}
\cos(a_c)
&=-1+\frac{1}{2}\left(
				\left(1+\frac{\sqrt{78}}{9}\right)^{1/3}
				+
				\left(1-\frac{\sqrt{78}}{9}\right)^{1/3}
				\right)\\
&=-0.23931...,
\end{split}
\end{equation}
and
\begin{equation}
a_c=1.8124....
\end{equation}
which is slightly smaller than 
$7\pi/12=1.8325...$.
The corresponding $x_c$ is
\begin{equation}
x_c=a_c/2=0.90622....
\end{equation}
This is the crossing point of the curve and the straight line $a=2x$.

The curve $h=0$ is apparently convex, although we don't give a proof.
Assuming the convexity, $a_c$ is the maximum value of $a$ on the curve $h=0$.
Since $a$ is the largest angle $a_\ell$  in this region,
the scalene relative equilibria exists for $\pi/2<a_\ell < a_c$.
Therefore, the scalene relative equilibria cannot have continuation to
the Euclidean plane. 
In \cite{Bengochea} the authors proved that any relative equilibria on the plane can be extended to spaces of constant curvature $\kappa$ when the parameter $\kappa$ is small. The above result shows that the inverse is not true, the dynamics on the sphere is much richer that on Euclidean spaces.

By the symmetry,
the crossing point of the curve and the straight line $a=x/2$ is
$(a,x)=(a_c/2,a_c)$.

Since $\cos(2y)$ is explicitly expressed by the function of $a$
as in (\ref{solOfCos2y}),
 we can exactly determine the scalene rigid rotator and the corresponding relative equilibrium
for given $a \in (\pi/2, a_c)$.
The exact values for $\cos(a)=\cos(\theta_2-\theta_1)=-1/8$ are shown
in the appendix \ref{secExactValues}.

\begin{figure}
   \centering
   \includegraphics[width=12cm]{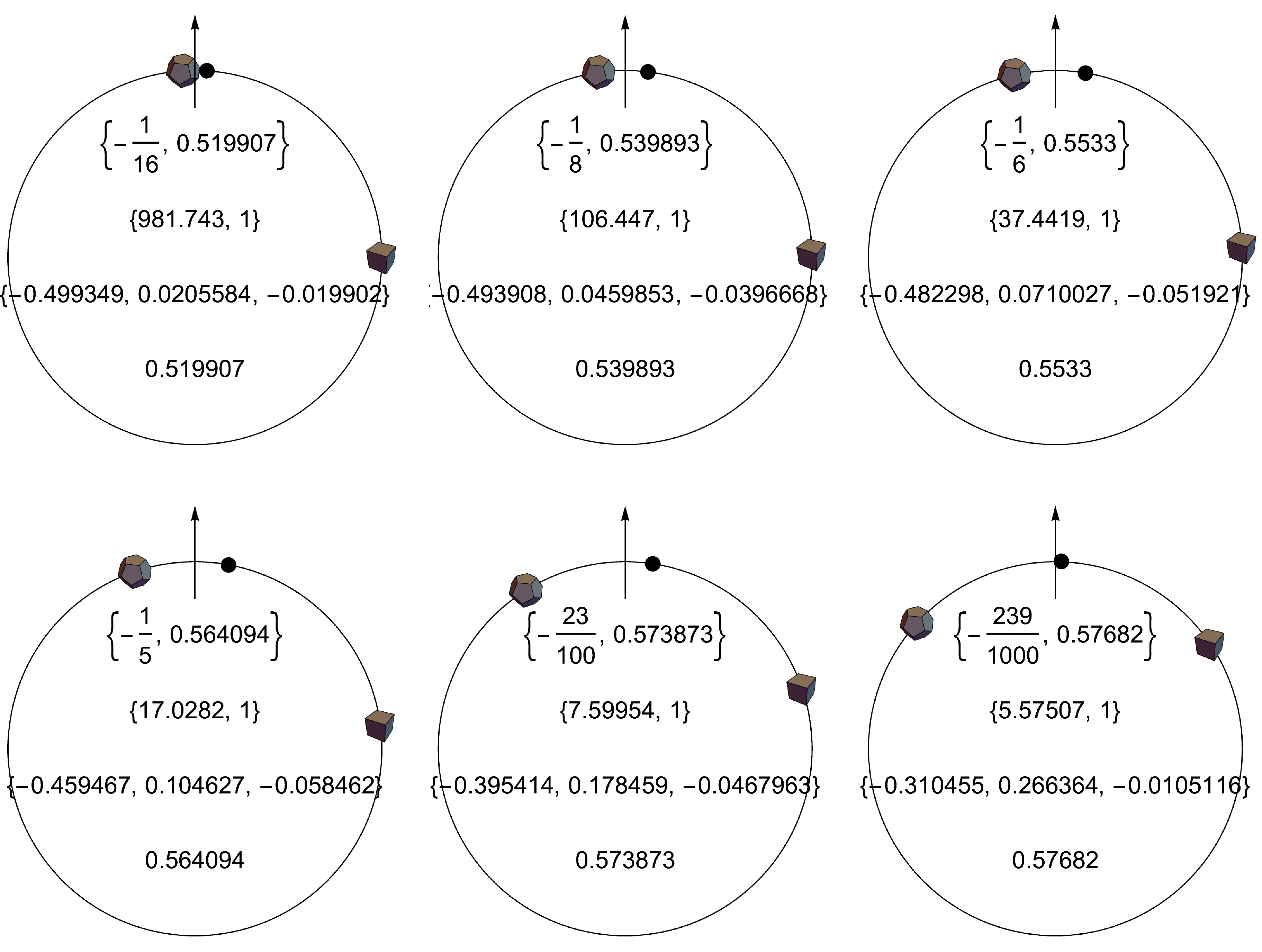} 
   \caption{
 Scalene relative equilibria.
   From top left to bottom right,
   $\cos\theta_{21}=-1/16$, $-1/8$, $-1/6$, $-1/5$, $-23/100$,
   and $-239/100$ with $y>0$.
   Each line in the picture represent
   $\{\cos \theta_{21}, \theta_{21}/\pi\}$,
   $\{\omega^2,s\}$,
   $\{\theta_1/\pi,\theta_2/\pi,\theta_3/\pi\}$,
   and $a_\ell/\pi$.
   For this pictures, $a_\ell=\theta_{21}$.
   The limit $\cos a_\ell \to 0$
   corresponds to the collision of $m_2$ and $m_3$.
 On the other hand,
 the limit $\cos a_\ell  \to \cos a_c=-0.23931...$
 corresponds  to the isosceles relative equilibrium with 
 $\theta_3=0, -\theta_1=\theta_2=a_c/2$.
 }
   \label{figScaleneRE}
\end{figure}

The scalene relative equilibria for $\pi/2<a_\ell<a_c$ and $y>0$ are shown
in Figure \ref{figScaleneRE}.
The limit for $a_\ell \to \pi/2$ is the collision of $m_2$ and $m_3$.
On the other hand, the limit for $a_\ell \to a_c$ is the isosceles 
relative  equilibrium with $\theta_3=0$ and $-\theta_1=\theta_2=a_c/2$.
The other scalene relative equilibria on the curve in the Figure \ref{figFbyXorY}
are generated by the exchange of masses.
(The $\pi$ rotation always occurs over time.)

\begin{remark}
Since we are working with the cotangent potential \eqref{cotangent pot}, which is the potential that is generally used for the study of the curved $3$ body problem,
when the curvature is positive, the previous discussion shows the existence of scalene Eulerian relative equilibria for the equal masses case. As far as we know this is the first time that scalene triangles are shown in this case.
\end{remark}


\section{Isosceles relative equilibria}
\label{secIsosceles}
In Section \ref{conditionForRotator},
it is shown (graphically) that four isosceles $RE$ exist for given $a=\theta_2-\theta_1$.
Moreover,
almost every
isosceles shape is a rigid rotator.

To express an isosceles triangle, we take $m_3$ 
at the mid point between 
$m_1$ and $m_2$
with
\begin{equation}
\theta=\theta_2-\theta_3=\theta_3-\theta_1,
\,\,\, 0<\theta<\pi.
\end{equation}

In this section, the following 
result will be proved.

\begin{theorem}[Isosceles configurations]
\label{propIsosceles}
Any isosceles triangle
 except $\theta=\pi/2$ is a rigid rotator. For the corresponding relative equilibria we have
\begin{itemize}
\item[i)] $\theta_3=0 \mod(\pi)$ for $\theta\in (0,2\pi/3) 
\setminus
\{\pi/2\},$
\item[ii)] $\theta_3=\pi/2 \mod(\pi)$ for $2\pi/3<\theta<\pi,$
\item[iii)]  $\theta_3$ is arbitrary for $\theta=2\pi/3.$
\end{itemize}
\end{theorem}

\begin{proof}
First we prove that 
almost every
isosceles triangle is a rigid rotator.

For an isosceles triangle, $G_{ij}$ and $F_{ij}$ are
\begin{equation}
\begin{split}
G_{12}&=\sin(4\theta),
\quad G_{23}=G_{31}=-\sin(2\theta),\\
F_{12}&=\frac{\sin(2\theta)}{|\sin(2\theta)|^3},
\quad F_{23}=F_{31}=-\frac{1}{\sin^2(\theta)}.
\end{split}
\end{equation}
Then 
$G_{12}-G_{23}=-(G_{31}-G_{12})
=\sin(4\theta)+\sin(2\theta)
=\sin(2\theta)(1+2\cos(2\theta))$,
and $F_{12}-F_{23}=-(F_{31}-F_{12})=\sin(2\theta)/|\sin(2\theta)|^3+1/\sin^2(\theta)$.
Therefore, the condition $f=0$ is trivially satisfied.
However, $F_{12}$ is infinity for $\theta=\pi/2$,
where the equations of motion for $m_1$ and $m_2$ are not defined.
So, $\theta=\pi/2$ should be excluded.

Now we will determine the corresponding relative equilibrium,
namely, determine $\omega^2$ and $\theta_k$.
We observe that
\begin{equation}
\label{equationForSandOmegaSquare}
\frac{s\omega^2}{2A}
=\frac{F_{12}-F_{23}}{G_{12}-G_{23}}
=\frac{1}{\sin(2\theta)(1+2\cos(2\theta))}
		\left(\frac{\sin(2\theta)}{|\sin(2\theta)|^3}
			+\frac{1}{\sin^2(\theta)}
		\right),
\end{equation}
with 
\begin{equation}
A=\sqrt{3+2\cos(2\theta)+\cos(4\theta)}=|1+2\cos(2\theta)|.
\end{equation}

For $\theta\ne \pi/3, 2\pi/3$, we have that $\cos(2\theta)\ne -1$ and then
\begin{equation}
\frac{s\omega^2}{2}\frac{1+2\cos(2\theta)}{|1+2\cos(2\theta)|}
=\frac{1}{|\sin(2\theta)|^3}+\frac{1}{\sin^2(\theta)\sin(2\theta)}.
\end{equation}
The right hand side is positive for $0<\theta<2\pi/3$,
and negative for $2\pi/3<\theta<\pi$.
Therefore,
\begin{equation}
\label{signFunctionForEqualMassIsosceles}
\frac{s(1+2\cos(2\theta))}{|1+2\cos(2\theta)|}
=
\begin{cases}
+1&\mbox{for }0<\theta<2\pi/3,\\
-1&\mbox{for }2\pi/3<\theta<\pi,
\end{cases}
\end{equation}
and 
\begin{equation}
\label{omegaSquareForEqualMassIsosceles}
\omega^2
=
\begin{cases}
+2\big(1/|\sin(2\theta)|^3
+1/(\sin^2(\theta)\sin(2\theta))\big)
&\mbox{for }0<\theta<2\pi/3,\\
-2\big(1/|\sin(2\theta)|^3
+1/(\sin^2(\theta)\sin(2\theta))\big)
&\mbox{for }2\pi/3<\theta<\pi.
\end{cases}
\end{equation}

By the translation formula (\ref{tranlationFormula2}) and (\ref{signFunctionForEqualMassIsosceles}),
\begin{equation}
\begin{split}
\sin(2\theta_3)&=0 ,\\
\cos(2\theta_3)&=\frac{s(1+2\cos(2\theta))}{|(1+2\cos(2\theta))|}
=
\begin{cases}
+1&\mbox{for }0<\theta<2\pi/3,\\
-1&\mbox{for }2\pi/3<\theta<\pi.
\end{cases}
\end{split}
\end{equation}

Namely,
\begin{equation}
\theta_3=
\begin{cases}
0 \mod (\pi)&\mbox{for }0<\theta<2\pi/3,\\
\pi/2 \mod (\pi) &\mbox{for }2\pi/3<\theta<\pi.
\end{cases}
\end{equation}

For $\theta=\pi/3$ or $2\pi/3$ ($\cos(2\theta)=-1$),
as we have discussed in the last paragraph of Section
\ref{secTranslation},
the equations of motion determine $\omega^2$ and $\theta_k$.
The result is that
$\omega^2=16/(3\sqrt{3})$, $\theta_3=0$ and $\theta_2=-\theta_1=\pi/3$ for $\theta=\pi/3$,
and
$\omega^2=0$ and $\theta_k$ are undetermined for $\theta=2\pi/3$.

Thus, we finally get
\begin{equation}
\label{finalResultForEqualMassIsosceles}
\begin{split}
&(\omega^2,\theta_3)\\
&=
\begin{cases}
+2\big(1/|\sin(2\theta)|^3
+1/(\sin^2(\theta)\sin(2\theta)),0 \!\!\!\mod \!\pi\big)\\
\hfill\mbox{for }\theta\in (0,2\pi/3) \setminus \{\pi/2\},\\
(0, \mbox{ undetermined})\hfill\mbox{for }\theta=2\pi/3,\\
-2\big(1/|\sin(2\theta)|^3
+1/(\sin^2(\theta)\sin(2\theta)),\pi/2\!\!\!\mod\!\pi\big)\\
\hfill\mbox{for }\theta\in (2\pi/3,\pi).
\end{cases}
\end{split}
\end{equation}
\end{proof}

We can easily check this result by direct calculations of the equations of motion
with $\theta_3=0$ or $\pi/2$ and $\theta=\theta_2-\theta_3=\theta_3-\theta_1$.

\section{Final remarks}
\label{secFinalRemarks}
Theorem \ref{propConditionForShape}
holds true for general masses, and generic potential $U(\cos\theta_{ij})$
in the same form,
with
\begin{equation}
\begin{split}
G_{ij}&=m_im_j\sin(2\theta_{ji}),\\
F_{ij}&=-m_im_j\sin(\theta_{ji})U'(\cos\theta_{ij}),
\end{split}
\end{equation}
and
\begin{equation}
U'(\cos\theta)=\frac{dU(\cos\theta)}{d\cos\theta}.
\end{equation}

The method described in this paper  is applicable to the system with repulsive force
by just replacing $U \to -U$ and $F_{ij}\to -F_{ij}$.
As the previous result, the same shape is the rigid rotator
with the same $\omega^2$ and $s \to -s$.
Therefore the corresponding relative equilibrium
has the angle $\theta_k \to \theta_k+\pi/2 \mod (\pi)$.

For the repulsive cotangent potential for equal masses case,
Theorem \ref{propIsosceles} holds just exchanging 
$\theta_3=0 \mod(\pi)$ and $\theta_3=\pi/2 \mod(\pi)$.

The method is also applicable to the 
three charged particles
with 
\begin{equation}
U=-\frac{e_ie_j\cos\theta_{ij}}{\sqrt{1-\cos^2\theta_{ij}}},
\end{equation}
where $e_k$ is the charge of the particle $k=1,2,3$.
For this case,
\begin{equation}
\begin{split}
G_{ij}&=m_im_j\sin(2\theta_{ji}),\\
F_{ij}&=-e_ie_j\frac{\sin(\theta_{ji})}{|\sin(\theta_{ji})|^3}.
\end{split}
\end{equation}
Therefore, for the classical three similar particles problem on $\mathbb{S}^2$, it
has the same relative equilibria
described in this paper
with total $90$ degree rotation.


\appendix
\section{Exact values for the scalene equilibrium \\ for
$\cos(\theta_{21})=-1/8$ and $y>0$}
\label{secExactValues}
For $\cos(a)=-1/8$,
$h=(512 \cos(4y)-15368 \cos(2y)+6513)/2048=0$.
The solution is
\begin{equation}
\cos(2y)=\frac{1921-441\sqrt{17}}{256} >0.
\end{equation}

Then, $\sin(a)>0$ and $\cos(2a)$, $\sin(2a)$ are
\begin{equation}
\begin{split}
&\cos\theta_{21}=\cos a=-\frac{1}{8},
\,\,\,\sin \theta_{21}=\sin a=\frac{3\sqrt{7}}{8}=0.9921...,\\
&\cos(2\theta_{21})=\cos(2a)
=-\frac{31}{32}
=-0.9687....\\
&\sin(2\theta_{21})=\sin(2a)
=-\frac{3\sqrt{7}}{32}
=-0.2480....
\end{split}
\end{equation}
And,
$\cos(a/2)$, $\sin(a/2)$ are
\begin{equation}
\cos(a/2)
=\frac{\sqrt{7}}{4},
\,\,\,\sin(a/2)=\frac{3}{4}.
\end{equation}

The value of $\cos y>0$, $\sin y>0$ are
\begin{equation}
\begin{split}
\cos y
=\frac{1}{16}\sqrt{\frac{7(311-63\sqrt{17})}{2}},
\,\,\,
\sin y=\frac{3}{16}\sqrt{\frac{-185+49\sqrt{17}}{2}}.
\end{split}
\end{equation}

Since $x=y+a/2$,
\begin{equation}
\begin{split}
\cos\theta_{31}=\cos x
&=\frac{1}{128}\Big(
	7\sqrt{622-126\sqrt{17}}
	-9\sqrt{-370+98\sqrt{17}}
	\Big)
	=0.1432...,\\
\sin \theta_{31}=\sin x
&=\frac{3\sqrt{7}}{128}\Big(
	\sqrt{622-126\sqrt{17}}
	+\sqrt{-370+98\sqrt{17}}
	\Big)
	=0.9896...
\end{split}
\end{equation}
Then, for $2\theta_{31}=2x$ we get
\begin{equation}
\begin{split}
&\cos(2\theta_{31})
=\frac{441 \sqrt{17}-63 \sqrt{26894 \sqrt{17}-110014}-1921}{2048}
=-0.9589...\\
&\sin(2\theta_{31})
=-\frac{3 \sqrt{7} \left(441 \sqrt{17}+\sqrt{26894 \sqrt{17}-110014}-1921\right)}{2048}
=0.2835...
\end{split}
\end{equation}

For $\theta_{32}=\theta_{31}-\theta_{21}$, we have
\begin{equation}
\begin{split}
&\cos(\theta_{32})
=\frac{1}{128} \left(9 \sqrt{98 \sqrt{17}-370}+7 \sqrt{622-126 \sqrt{17}}\right)
=0.9640...\\
&\sin(\theta_{32})
=\frac{-3}{128}\left(\sqrt{4354-882 \sqrt{17}}-\sqrt{686 \sqrt{17}-2590}\right)
=-0.2658...\\
&\cos(2\theta_{32})
=\frac{441 \sqrt{17}+63 \sqrt{26894 \sqrt{17}-110014}-1921}{2048}
=0.8586...,\\
&\sin(2\theta_{32})
=\frac{3 \sqrt{7} \left(441 \sqrt{17}-\sqrt{26894 \sqrt{17}-110014}-1921\right)}{2048}
=-0.5125....
\end{split}
\end{equation}

\begin{equation}
A
=\frac{3}{16} \sqrt{\frac{1}{2} \left(49 \sqrt{17}-153\right)}
=0.9283...
\end{equation}

$F_{ij}$ and $G_{ij}$ are
\begin{equation}
\begin{split}
&F_{12}=\frac{64}{63},\\
&F_{23}=\frac{-8192}{
	63\Big(\sqrt{311-63\sqrt{17}}-\sqrt{-185+49\sqrt{17}}
		\Big)^2
	},\\
&F_{31}=\frac{8192}{
	63\Big(\sqrt{311-63\sqrt{17}}+\sqrt{-185+49\sqrt{17}}
		\Big)^2
	},
\end{split}
\end{equation}
\begin{equation}
\begin{split}
&G_{12}=-\frac{3\sqrt{7}}{32},\\
&G_{23}=\frac{3 \sqrt{7} \left(441 \sqrt{17}-1921
	-\sqrt{26894 \sqrt{17}-110014}\right)}{2048},\\
&G_{31}=\frac{3 \sqrt{7} \left(441 \sqrt{17}-1921
	+\sqrt{26894 \sqrt{17}-110014}\right)}{2048}.
\end{split}
\end{equation}

Then,
\begin{equation}
(G_{12}-G_{23})(G_{31}-G_{12})
=\frac{63 \left(6503 \sqrt{17}-26815\right)}{16384}
\ne 0,
\end{equation}
\begin{equation}
\frac{F_{12}-F_{23}}{G_{12}-G_{23}}
=\frac{F_{31}-F_{12}}{G_{31}-G_{12}}
=\frac{16}{189} \sqrt{55614 \sqrt{17}+\frac{1605122}{7}}
>0.
\end{equation}
Therefore, $s=1$ and
\begin{equation}
\omega^2
=\frac{32}{63} \sqrt{5326 \sqrt{17}+\frac{153714}{7}}
=106.44...
\end{equation}

Since $A$ and $s$ are determined,
we can determine $\sin(2\theta_k)$,
\begin{equation}
\begin{split}
&\sin(2\theta_1)
=\frac{1}{64} \sqrt{\frac{7}{98 \sqrt{17}-306}} \left(441 \sqrt{17}-1857+\sqrt{26894 \sqrt{17}-110014}\right),\\
&\sin(2\theta_2)
=\frac{1}{64} \sqrt{\frac{7}{98 \sqrt{17}-306}} \left(-441 \sqrt{17}+1857+\sqrt{26894 \sqrt{17}-110014}\right),\\
&\sin(2\theta_3)
=-\frac{1}{32} \sqrt{1120-\frac{4361}{\sqrt{17}}}.
\end{split}
\end{equation}

So far, 
we get exact values for $\sin(2\theta_k)$, $\sin(\theta_i-\theta_j)$
and $\omega^2$.
Therefore, we can directly check that
this configuration satisfies the equations of motion (\ref{equationsOfMotion}).

\subsection*{Acknowledgements}
The second author (EPC) has been partially supported 
by Asociaci\'on Mexicana de Cultura A.C. and Conacyt-M\'exico Project A1S10112.

\end{document}